\documentclass[11pt]{amsart}
\usepackage{amsfonts,amssymb,amsmath,enumerate,amsthm,bbold,cite,leftidx}
     
    \newcommand{\BC}{{\bold {C}}}

     \newcommand{\BZ}{{\bold {Z}}}

    \newcommand{\CC}{{\mathcal {C}}}

    \newcommand{\fg}{{\mathfrak{g}}} \newcommand{\fh}{{\mathfrak{h}}}

    \newcommand{\fm}{{\mathfrak{m}}}

    \newcommand{\Aut}{{\mathrm{Aut}}}

\newtheorem{thm}{Theorem}[subsection]
\newtheorem{defn}[thm]{Definition}
\newtheorem{prop}[thm]{Proposition}
\newtheorem{rmk} [thm]{Remark}
\newtheorem{lem}[thm]{Lemma}
\newtheorem{cor}[thm]{Corollary}

\newcommand{\ol}{\overline}

    \newcommand{\pair}[1]{\langle {#1} \rangle}

    \newcommand{\Poincare}{Poincar\'{e}~}
    
    \numberwithin{equation}{section}

\begin{document}
\title{On the convergence of the Kac-Moody Correction Factor}
\author{Yanze Chen$^*$}
\address{Departmemt of Mathematics, the Hong Kong University of Science and Technology, Clear Water Bay, Kowloon, Hong Kong}
\email{ychenen@connect.ust.hk}
\thanks{$*$This research is supported by Hong Kong RGC grant 16301718}
\begin{abstract}
	The Kac-Moody correction factor, first studied by Macdonald in the affine case, corrects the failure of an identity found by Macdonald in finite-dimensional root systems in 1972. Subsequntly this factor appeared in several formulas in the affine or Kac-Moody analogue of  $p$-adic spherical theory for reductive groups. In this article we view the inverse of this correction factor as a function, prove the convergency and holomorphy of this function on a certain domain. 
\end{abstract}
\maketitle
\section{Introduction}
\subsection{Kac-Moody Correction Factor}
Let $(V,\Delta=\{a_1,\cdots,a_l\},V^*,\Delta^\vee)$ be a Kac-Moody root system, $\Phi_{re}$ be the set of real roots in $V$, $\Phi_{re,+}$ be the set of positive real roots, $W$ be the Weyl group with length function $\ell$. Then $$\Delta_{re,t}:=\prod_{a\in\Phi_{re,+}}(1-te^{-a})$$ is a unit in the ring $\BZ[[t]][[Q_-]]:=\BZ[[t]][[e^{-a_1},\cdots,e^{-a_l}]]$, and $$\frac{\Delta_{re,t}}{\Delta_{re,1}}=\prod_{a\in\Phi_{re,+}}\frac{1-te^{-a}}{1-e^{-a}}$$ is an element in the same ring by formally expanding all the denominators in terms of negative roots. There is no reasonable $W$-action on this ring, but we can formally define $w\left(\frac{\Delta_{re,t}}{\Delta_{re,1}}\right)$ by setting \begin{equation}
	w\left(\frac{\Delta_{re,t}}{\Delta_{re,1}}\right)=\left(\prod_{b\in\Phi(w)}\frac{t-e^{-b}}{1-te^{-b}}\right)\cdot\left(\frac{\Delta_{re,t}}{\Delta_{re,1}}\right)
\end{equation} where $\Phi(w)=\{b\in\Phi_{re,+}:w^{-1}b<0\}$. Note that $w\left(\frac{\Delta_{re,t}}{\Delta_{re,1}}\right)$ is an element in $\BZ[[t]][[Q_-]]$ since the set $\Phi(w)$ is finite.  The reason we make this definition is the following formal computation: \begin{align}
	&\prod_{a\in\Phi_{re,+}}\frac{1-te^{-wa}}{1-e^{-wa}}= \left(\prod_{a\in\Phi_{re,+},wa<0}\frac{1-te^{-wa}}{1-e^{-wa}}\right)\left(\prod_{a\in\Phi_{re,+},wa>0}\frac{1-te^{-wa}}{1-e^{-wa}}\right)\\\nonumber
	&=\left(\prod_{b\in\Phi_{re,+},w^{-1}b<0}\frac{1-te^{b}}{1-e^{b}}\right)\left(\prod_{b\in\Phi_{re,+},w^{-1}b>0}\frac{1-te^{-b}}{1-e^{-b}}\right)\\\nonumber
	&=\left(\prod_{b\in\Phi(w)}\frac{1-te^b}{1-e^b}\cdot\frac{1-e^{-b}}{1-te^{-b}}\right)\cdot\left(\frac{\Delta_{re,t}}{\Delta_{re,1}}\right)= \left(\prod_{b\in\Phi(w)}\frac{t-e^{-b}}{1-te^{-b}}\right)\cdot\left(\frac{\Delta_{re,t}}{\Delta_{re,1}}\right)
\end{align}
\par It can be proved that $\sum_{w\in W}w\left(\frac{\Delta_{re,t}}{\Delta_{re,1}}\right)$ is a well-defined element in the ring $\BZ[[t]][[Q_-]]$ (see \cite{MPW}), and the "constant term" of $\sum_{w\in W}w\left(\frac{\Delta_{re,t}}{\Delta_{re,1}}\right)$ is the \Poincare series $W(t):=\sum_{w\in W}t^{\ell(w)}$ of the Coxeter group $W$, i.e. if we set all the formal exponentials $e^{-a}$ to be zero, then the summation just gives $W(t)$. It is known that $W(t)\in1+t\BZ[[t]]$ is "rational", namely it is the quotient of two polynomials in $\BZ[t]$ (c.f. \cite{St}). In particular, $\sum_{w\in W}w\left(\frac{\Delta_{re,t}}{\Delta_{re,1}}\right)$ is a unit in the ring $\BZ[[t]][[Q_-]]$, so there exists $\fm\in\BZ[[t]][[Q_-]]$ such that $$\fm\sum_{w\in W}w\left(\frac{\Delta_{re,t}}{\Delta_{re,1}}\right)=W(t)$$ If the Kac-Moody root system $(V,\Delta)$ is of finite type, then Macdonald proved that $\fm=1$ in \cite{Mac1}. In the affine case, $\fm$ is first studied also by Macdonald in \cite{Mac2}, where he considered possibly non-reduced affine root systems and gave explicit formulas for $\fm$ case by case. Note that the computation is equivalent to the Macdonald constant term conjecture \cite{Mac3}, which is proved by Cherednik \cite{Ch}. In the general Kac-Moody case $\fm$ is called the \textbf{Kac-Moody correction factor}. The affine correction factor appears in the affine analogue of the Gindikin-Karpelevich formula \cite{BFK,BGKP}, the Macdonald formula \cite{BKP}, and the Casselman-Shalika formula \cite{P} in the works towards a generalization of $p$-adic spherical theory to affine Kac-Moody groups. Subsequently the Kac-Moody correction factor also appeared in the generalizations of the above works to arbitrary $p$-adic Kac-Moody groups \cite{BPGR1,BPGR2,PP}. 
\par The Kac-Moody correction factor $\fm$ was first studied in \cite{MPW} and later in \cite{LLO}. In \cite{MPW} the authors considered a formal product expansion $$\fm=\prod_{\lambda\in Q_-}\prod_{n\geq0}(1-t^ne^\lambda)^{-m(\lambda,n)}$$ of the correction factor, gave a generalized Peterson algorithm to compute $\fm$ recursively, and they showed that $m_\lambda(t)=\sum_{n\geq0}m(\lambda,n)t^n$ are polynomials and gave a closed formula for this polynomial. In the paper \cite{LLO} the authors proved that $\fm^{-1}$ times a certain infinite product over imaginary roots can be expressed as an infinite $\BZ[[t]]$-linear combination of irreducible characters of the corresponding Kac-Moody algebra, and they studied the coefficients. Note that in \cite{MPW,LLO}, the correction factor was treated as a formal power series.
\par In this paper we will study the inverse $\fm^{-1}$ of the correction factor from an analytic perspective.  In view of the further goal of generalizing the $p$-adic spherical theory and even the theory of automorphic forms to Kac-Moody groups, eventually this correction factor should be "evaluated" on the parameters of a principal series representation. Our main result is that $\fm^{-1}$ converges to a holomorphic function on the union of all $W$-translates of the domain $V+iC$ in $V\otimes\BC=V+iV$ (throughout this article, for subsets $A,B$ of $V$,  $A+iB$ denotes the set $\{x+iy\in V\otimes\BC=V+iV:x\in A,y\in B\}$), where $C$ is the dominant chamber in $V$, and it can be analytically continued to a holomorphic function on the domain $\Omega=V+iI^\circ$, where $I^\circ$ is the interior of the Tits cone. Note that the characters of irreducible highest weight modules of a symmetrizable Kac-Moody Lie algebra with dominant integral highest weights converges absolutely to a holomorphic function on the domain $\Omega$ \cite{KP,Lo}. This motivates us to study the inverse of the Kac-Moody correction factor as a complex analytic function on the domain $\Omega$. 
\subsection{The Inverse of the Correction Factor as an Analytic Function}In this section we will  describe the ideas of the proof of our main result. Note that we are going to change into the coroot version of the correction factor in accordance with the formulas in \cite{BKP}, \cite{BPGR2}, \cite{PP}, etc. 
\par By the definition (1.1), the inverse of the correction factor is $$\fm^{-1}=W(t)^{-1}\left(\sum_{w\in W}\prod_{b\in\Phi(w)}\frac{t-e^{-b^\vee}}{1-te^{-b^\vee}}\right)\cdot\left(\prod_{a\in\Phi_{re,+}}\frac{1-te^{-a^\vee}}{1-e^{-a^\vee}}\right)$$We view the formal exponential $e^{-a^\vee}$ for $a\in\Phi_{re,+}$ as a function on $V\otimes\BC$ whose value on $h\in V\otimes\BC$ is $e^{2\pi i\pair{h,a^\vee}}$, so we hope that \begin{equation}
	\CC(t,h)=\left(\sum_{w\in W}\prod_{b\in\Phi(w)}\frac{t-e^{2\pi i\pair{h,b^\vee}}}{1-te^{2\pi i\pair{h,b^\vee}}}\right)\left(\prod_{a\in\Phi_{re,+}} \frac{1-te^{2\pi i\pair{h,a^\vee}}}{1-e^{2\pi i\pair{h,a^\vee}}}\right)
\end{equation} defines a function on a subset of $\BC\times(V+iV)$. 
\par The main work of this article is to prove that $\CC(t,h)$ "defines" a complex analytic function for $t\in D_r$ and $h\in\Omega$, where $r$ is the radius of convergence of $W(t)$ and $D_r$ is the open disk $\{t\in\BC:|t|<r\}$. More precisely, we proved that the above defined $\CC(t,h)$ is absolutely convergent for $(t,h)\in D_r\times(V+iW\cdot C)$, uniformly on compact subsets, where $W\cdot C\subseteq V$ is the union of all $W$-translates of $C$. It is an open dense subset of $\Omega$, so $\CC(t,h)$ is a holomorphic function on $D_r\times(V+iW\cdot C)$. Then we proved that $\CC(t,h)$ can be analytically continued to a holomorphic function on the domain $D_r\times\Omega$. 
\par Now we begin to sketch the proof of the main result. We start by proving that $\CC(t,h)$ is absolutely convergent if $|t|<r$ and $h=x+iy\in V+iC$ where $C$ is the dominant chamber $$C=\{y\in V:\pair{y,a^\vee}>0,\,\forall a\in\Delta\}$$ Indeed, the second bracket of (1.3) can be bounded by a geometric progression, and the first bracket of (1.3) can be compared to $W(t)$ since as the depth of the root $b$ tends to $\infty$, we have $|e^{2\pi i\pair{h,b^\vee}}|\to0$ and thus $\frac{t-e^{2\pi i\pair{h,b^\vee}}}{1-te^{2\pi i\pair{h,b^\vee}}}\to t$. Since $|\Phi(w)|=\ell(w)$, the term corresponding to $w\in W$ in the summation of the first bracket is close to $t^{\ell(w)}$ when the positive root $b$ goes deep. Moreover, this convergence is uniform on compact sets, so $\CC(t,h)$ converges to a holomorphic function on the domain $D_r\times(V+iC)$.
\par Note that the formal computation in (1.2) implies that for $(t,h)\in D_r\times(V+iC)$ we have  $$\prod_{a\in\Phi_{re,+}} \frac{1-te^{2\pi i\pair{h,wa^\vee}}}{1-e^{2\pi i\pair{h,wa^\vee}}}=\left(\prod_{b\in\Phi(w)}\frac{t-e^{2\pi i\pair{h,b^\vee}}}{1-te^{2\pi i\pair{h,b^\vee}}}\right)\left(\prod_{a\in\Phi_{re,+}} \frac{1-te^{2\pi i\pair{h,a^\vee}}}{1-e^{2\pi i\pair{h,a^\vee}}}\right)$$ Namely, the infinite product on LHS is well-defined since RHS is absolutely convergent.  Moreover, for $(t,h)\in D_r\times (V+iC)$ we also have $$\sum_{w\in W}\prod_{a\in\Phi_{re,+}} \frac{1-te^{2\pi i\pair{h,wa^\vee}}}{1-e^{2\pi i\pair{h,wa^\vee}}}=\left(\sum_{w\in W}\prod_{b\in\Phi(w)}\frac{t-e^{2\pi i\pair{h,b^\vee}}}{1-te^{2\pi i\pair{h,b^\vee}}}\right)\left(\prod_{a\in\Phi_{re,+}} \frac{1-te^{2\pi i\pair{h,a^\vee}}}{1-e^{2\pi i\pair{h,a^\vee}}}\right)$$ namely the LHS summation is absolutely convergent uniformly on compact sets since RHS is so. Thus we have $$\CC(t,h)=\sum_{w\in W}\prod_{a\in\Phi_{re,+}} \frac{1-te^{2\pi i\pair{h,wa^\vee}}}{1-e^{2\pi i\pair{h,wa^\vee}}}$$ for $(t,h)\in D_r\times(V+iC)$. This is a series of summation over $W$, so by $W$-invariance $\CC(t,h)$ actually defines a holomorphic function on $D_r\times(V+iW\cdot C)$ 
\par However, $V+iW\cdot C\neq\Omega$ unless $(V,\Delta)$ is of finite type. In order to analytically continue the function $\CC(t,h)$ to the whole $\Omega$ we need to use the detailed description of $\Omega$ given by Looijenga in \cite{Lo}. We call a subset $X\subseteq\Delta$ \textbf{$W$-finite} if the subgroup $W_X$ in $W$ generated by fundamental reflections corresponding to simple roots in $X$ is finite. Following \cite{Lo}, for such a subset $X$ we define $$F_X=\{x\in V:\pair{x,a^\vee}=0\text{ for }a\in X,\,\pair{x,a^\vee}>0\text{ for }a\in \Delta-X \}$$ Then $\Omega$ is the union of $V+iW\cdot F_X$ for all $W$-finite $X\subseteq\Delta$. Moreover we have $F_X\subseteq\ol F_Y$ iff. $Y\subseteq X$, so $V+iW\cdot C$ is an open dense subset of $\Omega$. We then extend $\CC(t,h)$ to a holomorphic function on $D_r\times\Omega$. 
\par For simplicity, in this introduction we will only consider the situation $h\in V+iF_X$, since other cases can be treated by applying $W$-translations. The formula (1.3) can not be used to define a function $\CC(t,h)$ for $h\in V+iF_X$ because in the second bracket of (1.3) the denominator can be zero for some $h\in V+iF_X$. But it turns out that if we apply Macdonald's identity (see (1.5) below) for the finite-dimensional root system $(V,X)$, the poles will cancel. Let's return to the formal sum $$\sum_{w\in W}w\left(\prod_{a\in\Phi_{re,+}}\frac{1-te^{-a^\vee}}{1-e^{-a^\vee}}\right)$$ and consider another decomposition of it. For each $W$-finite $X\subseteq\Delta$, we first sum over $W_X$, and makle a formal computation as follows: 
\begin{align}
	&\sum_{w\in W}w\left(\prod_{a\in\Phi_{re,+}}\frac{1-te^{-a^\vee}}{1-e^{-a^\vee}}\right)=\sum_{w_1\in W^X}\sum_{w_2\in W_X}\prod_{a\in\Phi_{re,+}}\frac{1-te^{-w_1w_2a^\vee}}{1-e^{-w_1w_2a^\vee}}\\\nonumber&=\sum_{w_1\in W^X}\sum_{w_2\in W_X}\prod_{a\in \Phi_{X,+}}\frac{1-te^{-w_1w_2a^\vee}}{1-e^{-w_1w_2a^\vee}}\prod_{a\in\Phi_{re,+}-\Phi_{X,+}} \frac{1-te^{-w_1w_2a^\vee}}{1-e^{-w_1w_2a^\vee}}\\\nonumber&=\sum_{w_1\in W^X}w_1\left(\sum_{w_2\in W_X}\prod_{a\in\Phi_{X,+}}\frac{1-te^{-w_2a^\vee}}{1-e^{-w_2a^\vee}}\right)\cdot\prod_{a\in\Phi_{re,+}-\Phi_{X,+}}\frac{1-te^{-w_1a^\vee}}{1-e^{-w_1a^\vee}}
\end{align} where $W^X$ is the set of minimal length representatives for $W/W_X$, $\Phi_{X,+}$ is the set of positive roots for the root system $(V,X)$, and the last equality is because each $w_2\in W_X$ permutes the positive roots in $\Phi_{re,+}-\Phi_{X,+}$ since the positive roots inverted by $w_2$ all lies in $\Phi_{X,+}$. A main result in \cite{Mac1} is that \begin{equation}
		\sum_{w_2\in W_X}\prod_{a\in \Phi_{X,+}}\frac{1-te^{-w_2a^\vee}}{1-e^{-w_2a^\vee}}=W_X(t)	
 \end{equation} where $W_X(t)$ is the \Poincare polynomial of the Coxeter group $W_X$, which is a polynomial in $t$ with integral coefficients since $W_X$ is finite, thus the RHS of (1.5) is equal to \begin{align}
	&W_X(t)\sum_{w_1\in W^X} \prod_{a\in\Phi_{re,+}-\Phi_{X,+}}\frac{1-te^{-w_1a^\vee}}{1-e^{-w_1a^\vee}}\\\nonumber
	&=W_X(t)\left(\sum_{w\in W^X}\prod_{b\in\Phi(w)}\frac{t-e^{-b^\vee}}{1-te^{-b^\vee}}\right)\left(\prod_{a\in\Phi_{re,+}-\Phi_{X,+}}\frac{1-te^{-a^\vee}}{1-e^{-a^\vee}}\right)
\end{align}In this formal expression, every term in the denominator has no zero points in $V+iF_X$. So we define \begin{equation}
	\CC_X(t,h)=W_X(t)\left(\sum_{w\in W^X}\prod_{b\in\Phi(w)}\frac{t-e^{2\pi i\pair{h,b^\vee}}}{1-te^{2\pi i\pair{h,b^\vee}}}\right)\left(\prod_{a\in\Phi_{re,+}-\Phi_{X,+}}\frac{1-te^{2\pi i\pair{h,a^\vee}}}{1-e^{2\pi i\pair{h,a^\vee}}}\right)
\end{equation} We can prove the convergence of RHS of (1.7). Our proof is similar to that of the convergence of $\CC(t,h)$, but is technically more involved. In the proof we need to use a relation between the root systems $(V,\Delta)$ and $(V,X)$. Also the formal computation (1.6) yields 
 \begin{equation}
	\CC_X(t,h):= W_X(t)\sum_{w\in W^X}\prod_{a\in\Phi_{re,+}-\Phi_{X,+}}\frac{1-te^{2\pi i\pair{h,wa^\vee}}}{1-e^{2\pi i\pair{h,wa^\vee}}}
\end{equation}
\par Moreover, RHS of (1.7) is also convergent on each $V+iF_Y$ with $Y\subseteq X$, and the convergence is uniform on compact subsets of $V+i(\bigcup_{Y\subseteq X}F_Y)$. By absolute convergency we have $\CC(t,-)=\CC_X(t,-)$ on $V+iC$.  Thus $\CC_X(t,h)$ can be viewed as an extension of the function $\CC(t,h)$ to the boundary component $V+i(\bigcup_{Y\subseteq X}F_Y)$ of $I^\circ$. But this is not an open set, so we do not get an analytic continuation at this point. 
\par In order to get an analytic continuation, we recall from \cite{Lo} that the \textbf{star} of $F_X$ is the set $$S_X=\bigcup_{w\in W_X}\bigcup_{Y\subseteq X}wF_Y$$ This is open (\cite{Lo}), and we consider the open domain $V+iW\cdot S_X$ in $I^\circ$. By the $W$-invariance of $\CC_X(t,-)$, we can prove that $\CC_X(t,h)$ actually defines a $W$-invariant holomorphic function on $V+iW\cdot S_X$, hence is an analytic continuation of $\CC(t,-)$. By putting all the $\CC_X$ together, we have defined a holomorphic function on $\Omega$. 
\par This article will be organized as follows: in section 2 we will recall Looijenga's study of the structure of the interior of the Tits cone, and in section 3 we will prove the main theorem following the above recipe. 
\subsection{Acknowledgments} I would like to thank my supervisor Yongchang Zhu for suggesting this interesting problem and for enormous helpful discussions and encouragements during this work. 
\section{Kac-Moody Root Systems and Interior of the Tits Cone}
In this section we briefly review Kac-Moody root systems, and then review the description of the interior of the Tits cone following \cite{Lo}. We also recall some properties of Coxeter groups which will be used in section 3. 
\subsection{Kac-Moody Root Systems}
\begin{defn}
	A \textbf{Kac-Moody root system} is a quadruple $(V,\Delta,V^*,\Delta^\vee)$ where \begin{itemize}
		\item $V$ is a finite-dimensional real vector space. 
		\item $\Delta$ is a finite linearly independent subset of $V$. 
		\item $V^*$ is the linear dual of $V$. 
		\item $\Delta^\vee$ is a finite linearly independent subset of $V$ with a bijection $\vee:\Delta\to\Delta^\vee$. We will denote the image of $a\in\Delta$ under $\vee$ by $a^\vee$, and denote the image of a subset $X\subseteq\Delta$ under $\vee$ by $X^\vee$.  
	\end{itemize}satifsying the following axioms: \begin{enumerate}[(R1)]
		\item $\pair{a,a^\vee}=2$ for $a\in\Delta$, where $\pair{-,-}$ is the natural pairing between $V$ and $V^*$. 
		\item $\pair{a,b^\vee}$ is a non-positive integer for $a\neq b\in\Delta$. 
		\item $\pair{a,b^\vee}=0$ implies $\pair{b,a^\vee}=0$. 
	\end{enumerate}
\end{defn} If $\Delta\neq0$, we usually suppose $\Delta=\{a_1,\cdots,a_l\}$. Elements in $\Delta$ are sometimes called \textbf{simple roots}. Note that $(V^*,\Delta^\vee,V,\Delta)$ is also a Kac-Moody root system, called the \textbf{dual} of $(V,\Delta,V^*,\Delta^\vee)$. From now on, we will loosely refer to a Kac-Moody system only by the pair $(V,\Delta)$.
\begin{rmk}
	Irreducible Kac-Moody root systems up to isomorphism corresponds to generalized Cartan matrices (c.f.\cite{Kac}). 
\end{rmk}
\begin{defn}
	For $a\in\Delta$, define the corresponding \textbf{fundamental reflection} $s_a:V\to V$ by $s_a(v)=v-\pair{v,a^\vee}a$. The \textbf{Weyl group} $W$ of the Kac-Moody root system $(V,\Delta)$ is the subgroup of $\Aut(V)$ generated by the fundamental reflections corresponding to simple roots. We let $W$ acts on $V^*$ by $\pair{v,wl}=\pair{w^{-1}v,l}$ for $v\in V,l\in V^*,w\in W$. 
\end{defn}
\begin{defn}
	A \textbf{real root} in $V$ is an element which is $W$-conjugate to a simple root. Let $\Phi_{re}$ be the set of real roots in $V$, $\Phi_{re,+}$ be the subset of elements in $\Phi_{re}$ that can be written as a $\BZ_{\geq0}$-linear combination of simple roots, then a standard result states that $\Phi_{re}$ is the disjoint union of $\Phi_{re,+}$ and $-\Phi_{re,+}$. Elements in $\Phi_{re,+}$ are called \textbf{positive real roots}. 
\end{defn}
\begin{rmk}
	In \cite{Lo}, Looijenga used the word "root" for what is called real root here, because the imaginary roots do not play a role here. 
\end{rmk}
\subsection{The interior of the Tits cone}
\par Let $C=\{x\in V: \pair{x,a^\vee}>0$ for all simple roots $a\in\Delta\}$. Recall that the Tits cone is defined to be the subset $I=\bigcup_{w\in W}w\ol C$ of $V$. Let $I^\circ$ be the interior of the Tits cone. We will give a concrete description of $I^\circ$ in this section following \cite{Lo}. 
\par  For each subset $X\subseteq \Delta$ define $$F_X:=\{x\in V:\pair{x,a^\vee}=0\text{ for }a\in X,\,\pair{x,a^\vee}>0\text{ for }a\in \Delta-X \}$$ then clearly we have a disjoint decomposition $\ol C=\bigsqcup_{X\subseteq\Delta}F_X$ and $F_X\subseteq \ol{F_Y}$ iff $Y\subseteq X$. In particular we have $F_\phi=C$. A \textbf{facet} is defined to be a $W$-translate of $F_X$ for some $X\subseteq \Delta$.  For $X\subseteq\Delta$, the \textbf{star} of $F_X$, denoted $S_X$, is defined to be the union of all facets which contains $F_X$ in the closure.  We recall some results from \cite{Lo}:
\begin{lem}[\cite{Lo} (1.3)]
	If $w\in W$, $X\subseteq \Delta$ are such that $wF_X\cap\ol C\neq\phi$, then $w\in W_X$ and hence $w$ leaves $F_X$ pointwise fixed. In particular, $\ol C$ is a fundamental domain of the action of $W$ on $I$. 
\end{lem}
\begin{lem}
	For $X\subseteq\Delta$ s.t. $|W_X|<\infty$, we have $$S_X=\bigcup_{w\in W_X}\bigcup_{Y\subseteq X}wF_Y$$
\end{lem}
\begin{proof}
	If $F_X\subseteq\ol{wF_Y}=w\ol{F_Y}$, then $w^{-1}F_X\cap\ol C\neq\phi$, so by the above lemma, $w\in W_X$ and $w^{-1}F_X=F_X$, namely $F_X\subseteq\ol{F_Y}$, which implies $Y\subseteq X$. 
\end{proof}
In view of this lemma, we define the set $\Gamma_X:=\bigcup_{Y\subseteq X}F_Y$, then from definition we have $$\Gamma_X=\{x\in V:\pair{x,a^\vee}\geq0\text{ for }a\in X,\,\pair{x,a^\vee}>0\text{ for }a\in \Delta-X\}$$ And lemma 2.2.2 says that if $|W_X|<\infty$, then $S_X=W_X\cdot\Gamma_X$ is the union of all $W_X$-translates of $\Gamma_X$, so $$S_X=\{x\in V:\pair{x,a^\vee}>0\text{ for }a\in \Delta-X\}$$ In particular, $S_X$ is an open subset of $I^\circ$. 
\begin{lem}[\cite{Lo} (1.14)] $I^\circ$ is a disjoint union of facets with finite stabilizers in $W$. Namely, we have 
	$$I^\circ=\bigcup_{w\in W}\bigcup_{X\subseteq\Delta:|W_X|<\infty}wF_X$$ where $W_X$ is the subgroup of $W$ generated by simple reflections corresponding to elements in $X$. Moreover, $w_1F_X=w_2F_X$ iff. $w_1W_X=w_2W_X$. 
\end{lem}
\begin{proof}
	This follows from the above lemma. 
\end{proof}
\begin{cor}
	Every point in $\Omega$ is a $W$-translate of a point in $V+iF_X$ for some $W$-finite subset $X\subseteq\Delta$. 
\end{cor}
\begin{defn}
	A subset $X$ of $\Delta$ is called \textbf{$W$-finite} if $|W_X|<\infty$.
\end{defn}
\subsection{\Poincare Series of a Coxeter Group}Let $(W,S)$ be a Coxeter system with length function $\ell$, with the set $S$ of simple reflections in $W$ finite. For any subset $V\subseteq W$, we define $$V(t):=\sum_{w\in V}t^{\ell(w)}$$ the \textbf{\Poincare series} of $V$. We recall the following fundamental result in the theory of Coxeter groups: 
\begin{prop}[Minimal Length Representatives]
	For any $X \subseteq S$, let $W_X$ be the corresponding \textbf{parabolic subgroup} generated by simple reflections in $X$, define $$W^X:=\{w\in W:\ell(ws)>\ell(w),\,\forall s\in X\}$$ Then $W^X$ is a set of coset representatives of $W/W_X$. Each element $w'\in W^X$ is the unique element of minimal length in the coset $w'W_X$. For arbitrary $w'\in W^X,\,w\in W_X$, we have $\ell(w'w)=\ell(w')+\ell(w)$. 
\end{prop}
\begin{prop}[\cite{St}]
	$W(t)\in 1+t\BZ[[t]]$ is rational, namely it is a quotient of two polynomials in $\BZ[t]$.  
\end{prop}
\section{Correction Factor as Complex Analytic Functions}
\subsection{Convergence of the $W$-summation with imaginary part in the dominant cone}We first prove locally uniform convergence of the function $$\CC(t,h)=\left(\sum_{w\in W}\prod_{b\in\Phi(w)}\frac{t-e^{2\pi i\pair{h,b^\vee}}}{1-te^{2\pi i\pair{h,b^\vee}}}\right)\left(\prod_{b\in\Phi_{re,+}}\frac{1-te^{2\pi i\pair{h,b^\vee}}}{1-e^{2\pi i\pair{h,b^\vee}}}\right)$$ for $t\in D_r$ and $h$ in the region $V+iC$. Then we prove that it also converges on $W$-translates of $V+iC$ by using another expression of $\CC(t,h)$ in terms of a summation over $W$. 
\begin{prop}
	Let $r$ be the radius of convergence of $W(t)$. For $t\in D_r$ and $h=x+iy\in V+iC$, $$\CC(t,h)=\left(\sum_{w\in W}\prod_{b\in\Phi(w)}\frac{t-e^{2\pi i\pair{h,b^\vee}}}{1-te^{2\pi i\pair{h,b^\vee}}}\right)\left(\prod_{b\in\Phi_{re,+}}\frac{1-te^{2\pi i\pair{h,b^\vee}}}{1-e^{2\pi i\pair{h,b^\vee}}}\right)$$ is absolutely convergent uniformly on compact subsets of $D_r\times(V+iC)$, defining a holomorphic function on the domain $D_r\times(V+iC)$. 
\end{prop}
\begin{proof}
	 First note that $r\leq1$ since the series for $W(t)$ diverges at $t=1$. By $h=x+iy\in V+iC$ we have $$|e^{2\pi i\pair{h,b^\vee}}|=e^{-2\pi\pair{y,b^\vee}}<1\text{ for }b\in\Phi_{re,+}$$ so the denominators are all non-zero. We first prove the convergence of the infinite product $$\prod_{b\in\Phi_{re,+}}\frac{1-te^{2\pi i\pair{h,b^\vee}}}{1-e^{2\pi i\pair{h,b^\vee}}}$$  Recall that an infinite product $\prod_{i\in I}(1-a_i)$ is called \textbf{absolutely convergent} if the summation $\sum_{i\in I}|a_i|$ is convergent. If an infinite product is absolutely convergent, then we can freely change the order of the product without affecting the limit of the product. So we consider the infinite product $$\prod_{b\in\Phi_{re,+}}(1-te^{2\pi i\pair{h,b^\vee}})$$ whose absolutely convergence is guaranteed as long as the summation $$\sum_{b\in\Phi_{re,+}}|te^{2\pi i\pair{h,b^\vee}}|=t\sum_{b\in\Phi_{re,+}}e^{-2\pi\pair{y,b^\vee}}$$ is convergent. Since every positive real coroot is a finite sum of the simple coroots in $\Delta^\vee=\{a_1^\vee,\cdots,a_l^\vee\}$, we have \begin{align*}
	&\sum_{b\in\Phi_{re,+}}e^{-2\pi\pair{y,b^\vee}}\leq\sum_{n_1,\cdots,n_l\in\BZ_{\geq0}}e^{-2\pi\pair{y,n_1a_1^\vee+\cdots+n_la_l^\vee}}\\&=\prod_{i=1}^l\left(\sum_{n_i=0}^\infty e^{-2\pi \pair{y,a_i^\vee}n_i}\right)=\prod_{i=1}^l\frac{1}{1-e^{-2\pi\pair{y,a_i^\vee}}}
\end{align*} Hence the infinite product $\prod_{b\in\Phi_{re,+}}(1-te^{2\pi i\pair{h,b^\vee}})$ is absolutely convergent for arbitrary $t\in\BC$ and $h\in V+iC$. For $h$ varying in compact subsets of $V+iC$, the imaginary part $y$ of $h$ will vary in a compact subset of $C$, so the linear functions $\pair{y,a_i^\vee}$ have maximum and minimum and the summation $\sum_{b\in\Phi_{re,+}}|te^{2\pi i\pair{h,b^\vee}}|$ is thus uniformly absolutely convergent on compact subsets of $D_r\times(V+iC)$. By taking $t=1$, we see the product $$\prod_{b\in\Phi_{re,+}}\frac{1-te^{2\pi i\pair{h,b^\vee}}}{1-e^{2\pi i\pair{h,b^\vee}}}$$ is absolutely convergent uniformly on compact subsets because the denominators are non-zero. 
\par Next we prove the absolutely convergence of the summation $$\sum_{w\in W}\prod_{b\in\Phi(w)}\frac{t-e^{2\pi i\pair{h,b^\vee}}}{1-te^{2\pi i\pair{h,b^\vee}}}$$
The idea of the proof is, as the depth of the positive root $b$ grow large, $|e^{2\pi i\pair{h,b^\vee}}|$ becomes small, then $\frac{t-e^{2\pi i\pair{h,b^\vee}}}{1-te^{2\pi i\pair{h,b^\vee}}}$ is close to $t$, so we can try to compare the above sum with $W(t)=\sum_{w\in W}t^{\ell(w)}$. \par So we fix $0<r_0<r_1<r$, let $\ol D_{r_0}$ be the closed disk $\ol D_{r_0}=\{t\in \BC:|t|\leq r_0\}$, let $R$ be a compact subset of $C$, we will prove the above summation is uniformly absolutely convergent for $t\in\ol D_{r_0}$ and $h\in V+iR$. Since $\lim_{z\to0}\frac{t-z}{1-tz}=t$, there exists $c>0$ s.t. for all $|z|<c$ and $t\in \ol D_{r_0}$, we have $$\left|\frac{t-z}{1-tz}\right|<r_1$$ Take $N=-\frac{\log c}{2\pi}$. Note that each positive real coroot is a finite $\BZ_{\geq0}$-linear combination of simple coroots, and for each simple coroot $a_i^\vee$ we have $\min_{y\in R}\pair{y,a_i^\vee}>0$, there exists a finite subset $S\subseteq\Phi_{re,+}$ of positive roots s.t. $\pair{y,b^\vee}>N$ for all $b\in\Phi_{re,+}-S$ and $y\in R$. Then for $b\in\Phi_{re,+}-S$ and $y\in R$ we have $|e^{2\pi i\pair{x,b^\vee}-2\pi \pair{y,b^\vee}}|<c$ and  $$\left|\frac{t-e^{2\pi i\pair{h,b^\vee}}}{1-te^{2\pi i\pair{h,b^\vee}}}\right|=\left|\frac{t-e^{2\pi i\pair{x,b^\vee}-2\pi\pair{y,b^\vee}}}{1-te^{2\pi i\pair{x,b^\vee}-2\pi \pair{y,b^\vee}}}\right|<r_1$$
	\par Since $S$ is finite, there exists $A>0$ s.t. for all $b\in S$ we have $$\left|\frac{t-e^{2\pi i\pair{h,b^\vee}}}{1-te^{2\pi i\pair{h,b^\vee}}}\right|\leq A$$ So when $\ell(w)>S$, the term $$\prod_{b\in\Phi(w)}\left|\frac{t-e^{2\pi i\pair{x,b^\vee}-2\pi\pair{y,b^\vee}}}{1-te^{2\pi i\pair{x,b^\vee}-2\pi \pair{y,b^\vee}}}\right|$$ is bounded by $$r_1^{\ell(w)-|S|}A^{|S|}=r_1^{-|S|}A^{|S|}r_1^{\ell(w)}$$ uniformly for $|t|\leq r_0$ and $y\in R$, since each term is a product of $\ell(w)$ fractions, all but $|S|$ of them are smaller than $r_0$, the rest are smaller than $1$. The summation $$\sum_{w\in W}r_1^{\ell(w)}$$ is convergent since $0<r_1<r$. Thus the summation $$\sum_{w\in W,\ell(w)>|S|}\prod_{b\in\Phi(w)}\left|\frac{t-e^{2\pi i\pair{x,b^\vee}-2\pi\pair{y,b^\vee}}}{1-te^{2\pi i\pair{x,b^\vee}-2\pi \pair{y,b^\vee}}}\right|$$ is  uniformly convergent on $\ol D_{r_0}\times (V+iR)$ by Weierstrass test. So we proved that the summation $\CC(t,h)$ is absolutely convergent uniformly on compact subsets on $D_r\times(V+iC)$, hence defining a holomorphic function on this domain.\end{proof}
\begin{lem}
	For $t\in D_r$ and $h=x+iy\in V+iC$, the sum $$\sum_{w\in W}\prod_{a\in\Phi_{re,+}} \frac{1-te^{2\pi i\pair{h,wa^\vee}}}{1-e^{2\pi i\pair{h,wa^\vee}}}$$ is absolutely convergent uniformly on compact subsets of $D_r\times(V+iC)$, and is equal to $\CC(t,h)$. 
\end{lem}
\begin{proof}
	For each $w\in W$ we first prove that the term $$\prod_{a\in\Phi_{re,+}}\frac{1-te^{2\pi i\pair{h,wa^\vee}}}{1-e^{2\pi i\pair{h,wa^\vee}}}$$ is absolutely convergent uniformly on compact subsets of $D_r\times(V+iC)$, hence defines a holomorphic function on this domain. By the formal computation (1.2) we have
	\begin{align*}
		&\prod_{a\in\Phi_{re,+}}\frac{1-te^{2\pi i\pair{h,wa^\vee}}}{1-e^{2\pi i\pair{h,wa^\vee}}}=\left(\prod_{a\in\Phi_{re,+},wa<0}\frac{1-te^{2\pi i\pair{h,wa^\vee}}}{1-e^{2\pi i\pair{h,wa^\vee}}}\right)\left(\prod_{a\in\Phi_{re,+},wa>0}\frac{1-te^{2\pi i\pair{h,wa^\vee}}}{1-e^{2\pi i\pair{h,wa^\vee}}}\right)\\&=\left(\prod_{b\in\Phi_{re,+},w^{-1}b<0}\frac{1-te^{-2\pi i\pair{h,b^\vee}}}{1-e^{-2\pi i\pair{h,b^\vee}}}\right)\left(\prod_{b\in\Phi_{re,+},w^{-1}b>0}\frac{1-te^{2\pi i\pair{h,b^\vee}}}{1-e^{2\pi i\pair{h,b^\vee}}}\right)\\&=\left(\prod_{b\in\Phi(w)}\frac{1-te^{-2\pi i\pair{h,b^\vee}}}{1-e^{-2\pi i\pair{h,b^\vee}}}\cdot\frac{1-e^{2\pi i\pair{h,b^\vee}}}{1-te^{2\pi i\pair{h,b^\vee}}}\right)\left(\prod_{b\in\Phi_{re,+}}\frac{1-te^{2\pi i\pair{h,b^\vee}}}{1-e^{2\pi i\pair{h,b^\vee}}}\right)\\&=\left(\prod_{b\in\Phi(w)}\frac{t-e^{2\pi i\pair{h,b^\vee}}}{1-te^{2\pi i\pair{h,b^\vee}}}\right)\left(\prod_{b\in\Phi_{re,+}}\frac{1-te^{2\pi i\pair{h,b^\vee}}}{1-e^{2\pi i\pair{h,b^\vee}}}\right)
	\end{align*}
Since RHS is absolutely convergent uniformly on compact subsets of $D_r\times(V+iC)$ by proposition 3.1.1, so is the LHS. Moreover, we have $$\sum_{w\in W}\prod_{a\in\Phi_{re,+}} \frac{1-te^{2\pi i\pair{h,wa^\vee}}}{1-e^{2\pi i\pair{h,wa^\vee}}}=\left(\sum_{w\in W}\prod_{b\in\Phi(w)}\frac{t-e^{2\pi i\pair{h,b^\vee}}}{1-te^{2\pi i\pair{h,b^\vee}}}\right)\left(\prod_{b\in\Phi_{re,+}}\frac{1-te^{2\pi i\pair{h,b^\vee}}}{1-e^{2\pi i\pair{h,b^\vee}}}\right)$$ also because the absolutely convergency of RHS, so LHS is equal to $\CC(t,h)$ and has the same convergency properties. 
\end{proof}
\begin{cor}
	Let $r$ be the radius of convergence of $W(t)$. For $t\in D_r$ and $h=x+iy$ with $y\in W\cdot C$, the sum $$\CC(t,h)=\sum_{w\in W}\prod_{a\in\Phi_{re,+}}\frac{1-te^{2\pi i\pair{h,wa^\vee}}}{1-e^{2\pi i\pair{h,wa^\vee}}}$$ is absolutely convergent uniformly on compact subsets of $D_r\times (V+iW\cdot C)$, defining a holomorphic function on the domain $D_r\times(V+iW\cdot C)$. 
\end{cor}
\begin{proof}
\par 	Write $h=w'^{-1}h_0$ with $h_0=x_0+iy_0$ for $x_0\in V,y_0\in C$. Then we have $$\sum_{w\in W}\prod_{a\in\Phi_{re,+}}\frac{1-te^{2\pi i\pair{h,wa^\vee}}}{1-e^{2\pi i\pair{h,wa^\vee}}}=\sum_{w\in W}\prod_{a\in\Phi_{re,+}}\frac{1-te^{2\pi i\pair{h_0,w'wa^\vee}}}{1-e^{2\pi i\pair{h_0,w'wa^\vee}}}$$ which is equal to a change of order of the summation $$\sum_{w\in W}\prod_{a\in\Phi_{re,+}}\frac{1-te^{2\pi i\pair{h_0,wa^\vee}}}{1-e^{2\pi i\pair{h_0,wa^\vee}}}$$ which is absolutely convergent by lemma 3.1.2. So one can change the order of summation. The rest is clear. 
\end{proof}
\subsection{A lemma on Kac-Moody root systems}We have defined a holomorphic function $\CC(t,h)$ on $D_r\times(V+iW\cdot C)$. The next task is to find an analytic continuation of $\CC(t,h)$ to $D_r\times\Omega$. We will use a factorization similar to (3.1) to define the candidate for the analytic continuation on the "boundary components", and prove the convergence of it. Our method is similar to the proof of prop 3.1.1, but is more techinical. The following lemma is crucial for the proof of convergence. 
\begin{lem}
	\begin{enumerate}
		\item For $W$-finite $X\subseteq\Delta$, assume $\Delta=\{a_1,\cdots,a_l\}$ and $X=\{a_1,\cdots,a_k\}$, then for each tuple $(p_{k+1},\cdots,p_l)$ of non-negative integers, there are only finitely many choices of non-negative integers $n_1,\cdots,n_k$ s.t. $n_1a_1+\cdots+n_ka_k+p_{k+1}a_{k+1}+\cdots+p_la_l$ is a root in $\Phi_{re,+}$. Moreover, this number is of polynomial growth in $p_{k+1},\cdots,p_l$. 
		\item For any compact subset $R\subseteq\Gamma_X$ and $N>0$ there exists a finite subset $S\subseteq\Phi_{re,+}^\vee$ of positive coroots such that $\pair{y,b^\vee}>N$ for all $b^\vee\in\Phi_{re,+}^\vee-S$ and $y\in R$. 
	\end{enumerate}
\end{lem}
\begin{proof}
		\textit{(1).} Let $$\fg=\fh\oplus\left(\bigoplus_{a\in\Phi}\fg_a\right)$$ be the Kac-Moody Lie algebra corresponding to the Kac-Moody root system $(V,\Delta,V^*,\Delta^\vee)$, let $$\fg_X=\fh_X\oplus\left(\bigoplus_{a\in\Phi_X}\fg_a\right)$$ be the Levi subalgebra corresponding to $X$. Since $|W_X|<\infty$, $\fg_X$ is a finite-dimensional semi-simple Lie algebra. We consider the subspace $$M=\bigoplus_{n_1,\cdots,n_k\in\BZ_{\geq0}}\fg_{n_1a_1+\cdots+n_ka_k+p_{k+1}a_{k+1}+\cdots+p_la_l}$$ of $\fg$. We need to prove that $M$ is actually finite-dimensional. 
			\par Clearly M is an integrable $\fg_X$-module, thus by proposition 3.8 (b) in [Kac], each element in $M$ lies in a finite-dimensional $\fg_X$-submodule of $M$, so $M$ is a direct sum of finite-dimensional irreducible lowest weight $\fg_X$-modules with lowest weight anti-dominant and integral. Let $\lambda\in\fh_X$ be the restriction of the linear functional $p_{k+1}a_{k+1}+\cdots+p_la_l$ on $\fh^*$ to $\fh_X^*$, then the $\fh_X$-weight of $\fg_{n_1a_1+\cdots+n_ka_k+p_{k+1}a_{k+1}+\cdots+p_la_l}$ is $\lambda+n_1a_1+\cdots+n_ka_k$. We claim that for any fixed $\lambda$, there exists only finitely many choices of $n_1,\cdots,n_k\in\BZ_{\geq0}$ s.t. $\lambda+n_1a_1+\cdots+n_ka_k$ is anti-dominant: indeed, let $A_X$ be the Cartan matrix of $\fg_X$, which is of finite type, then $\lambda+n_1a_1+\cdots+n_ka_k$ is anti-dominant means that $$n_1\pair{a_i^\vee,a_1}+\cdots+n_k\pair{a_i^\vee,a_k}\leq-\pair{a_i^\vee,\lambda}$$ for $i=1,\cdots,k$, thus (in this step we used that $n_i\geq0$) \begin{equation}
				\begin{pmatrix}
				n_1 & n_2 & \cdots & n_k
			\end{pmatrix} A_X\begin{pmatrix}
				n_1\\n_2\\\cdots\\n_k
			\end{pmatrix}\leq\sum_{i=1}^k-n_i\pair{a_i^\vee,\lambda}
			\end{equation} Since $A_X$ is of finite type, the quadratic form $v\mapsto v^TA_Xv$ is positive-definite, the number of integer solution of this inequality is finite. So our claim is verified, which implies that $M$ has only finitely many anti-dominant weights, each of which is of finite multiplicity because the root multiplicities of $\fg$ are finite. So $M$ can have only finitely many irreducible lowest weight submodules, which together with the assertion at the beginning of this paragraph implies that $M$ is finite-dimensional. 
			\par To see the dimension of $M$ is of polynomial growth in $p_{k+1},\cdots,p_l$, we first note that the dimension of $M$ is bounded by the sum of dimensions of all the lowest weight modules with lowest weights the solutions of (3.1). Recall that we demand all the $n_i$'s to be non-negative integers. There exists a constant $C>0$, depending only on $A_X$,  such that $$C\max\{n_1,\cdots,n_k\}^2\leq \begin{pmatrix}
				n_1 & n_2 & \cdots & n_k
			\end{pmatrix} A_X\begin{pmatrix}
				n_1\\n_2\\\cdots\\n_k
			\end{pmatrix}$$ Also we have $$\sum_{i=1}^k -n_i\pair{\lambda,a_i^\vee}\leq\max\{n_1,\cdots,n_k\}\cdot\sum_{i=1}^k|\pair{\lambda,a_i^\vee}|$$ So the solutions of (3.1) are contained in the solutions of $$C\max\{n_1,\cdots,n_k\}^2\leq\left(\sum_{i=1}^k|\pair{\lambda,a_i^\vee}|\right)\cdot \max\{n_1,\cdots,n_k\}$$ whose solutions can be given by $$0\leq n_i\leq C^{-1}\sum_{i=1}^k|\pair{\lambda,a_i^\vee}|,\,i=1,\cdots,k$$ Thus $\dim M$ is bounded by \begin{equation}
				\sum_{0\leq n_1,\cdots,n_k\leq C^{-1}\sum_{i=1}^k|\pair{\lambda,a_i^\vee}|}\prod_{\alpha\in\Phi_{X,+}}\frac{\pair{\rho_X-\lambda-n_1a_1-\cdots-n_ka_k,\alpha^\vee}}{\pair{\rho_X,\alpha^\vee}}
			\end{equation} Here we used the Weyl dimension formula, namely the dimension of the irreducible lowest weight module of $\fg_X$ with anti-dominant integral lowest weight $\mu$ is $$\prod_{\alpha\in\Phi_{X,+}}\frac{\pair{\rho_X-\mu,\alpha^\vee}}{\pair{\rho_X,\alpha^\vee}}$$ where $\Phi_{X,+}$ is the set of positive roots of $\fg_X$, $\rho_X$ is half of the sum of the positive roots in $\Phi_{X,+}$. Clearly (3.2) is a polynomial of $\lambda$ and $C^{-1}\sum_{i=1}^k|\pair{\lambda,a_i^\vee}|$. Since $\lambda$ is of linear dependence to $p_{k+1},\cdots,p_l$, $\dim M$ can be bounded by a polynomial in $p_{k+1},\cdots,p_l$. 
 		\paragraph{\textit{(2)}} We write each positive coroot $b^\vee=n_1a_1^\vee+\cdots+n_la_l^\vee$ with $n_1,\cdots,n_l\in\BZ_{\geq0}$, then for $y\in\Gamma_X$ we have $$\pair{y,b^\vee}=\sum_{i=1}^l n_i\pair{y,a_i^\vee}\geq\sum_{i=k+1}^l n_i\pair{y,a_i^\vee}$$ with each $\pair{y,a_i^\vee}>0$. Since $R$ is compact, the functions $y\mapsto\pair{y,a_i^\vee}$ has maximum and minimum for $y\in R$, let $m_i(R)$ be its minimum, then $m_i(R)>0$ for $i=k+1,\cdots,l$. There exists only finitely many choices of non-negative integers $p_{k+1},\cdots,p_l$ s.t. $\sum_{i=k+1}^l n_im_i(R)\leq N$, let $S$ be the set of all positive coroots $a^\vee=n_1a_1^\vee+\cdots+n_ka_k^\vee+p_{k+1}a_{k+1}^\vee+\cdots+p_la_l^\vee$ such that $p_{k+1},\cdots,p_l$ is one of the previous choices. By part (1) of this lemma, $S$ is finite, and for any positive coroot $b^\vee=n_1a_1^\vee+\cdots+n_ka_k^\vee+n_{k+1}a_{k+1}^\vee+\cdots+n_la_l^\vee$ which does not belong to $S$ and $y\in R$ we have $$\pair{y,b^\vee}=\sum_{i=1}^l n_i\pair{y,a_i^\vee}\geq\sum_{i=k+1}^l n_i\pair{y,a_i^\vee}\geq\sum_{i=k+1}^l n_im_i(R)>N$$ Hence the result. 
\end{proof}
\subsection{Analytic continuation to $\Omega$}In this section we first define an analytic continuation $\CC_X(t,-)$ of $\CC(t,-)$ to the open sets $V+iW\cdot S_X$ for the $W$-finite subsets $X\subseteq\Delta$. We prove that $\CC_X(t,-)$ is $W$-invariant, so by putting the functions $\CC_X$ together, we get a holomorphic function on $\Omega$.
\begin{prop}
		Let $X\subseteq\Delta$ be a $W$-finite subset. For $t\in\BC,h=x+iy\in V+i\Gamma_X$, the infinite product $$\prod_{a\in\Phi_{re,+}}(1-te^{2\pi i\pair{h,a^\vee}})$$ as a function in $t$ and $h$ is absolutely convergent uniformly on compact subsets in $\BC\times(V+i\Gamma_X)$. 
\end{prop}
\begin{proof}
		We still assume $X=\{a_1,\cdots,a_k\}\subseteq \Delta=\{a_1,\cdots,a_l\}$. We need to prove the absolutely convergence of the summation $$\sum_{a\in\Phi_{re,+}}te^{2\pi i\pair{x,a^\vee}-2\pi\pair{y,a^\vee}}$$ namely we need to prove $$\sum_{a\in\Phi_{re,+}}e^{-2\pi\pair{y,a^\vee}}$$ is convergent, uniformly for $y$ in compact subsets of $\Gamma_X$.
		\par So let $R$ be a compact subset of $\Gamma_X$. Each positive root is a $\BZ_{\geq0}$-linear combination of simple roots, so for $y\in\Gamma_X$ we have \begin{align*}
			&\sum_{a\in\Phi_{re,+}}e^{-2\pi\pair{y,a^\vee}}\\&=\sum_{\substack{n_1,\cdots,n_k,p_{k+1},\cdots,p_l\in\BZ_{\geq0}\\n_1a_1^\vee+\cdots+n_ka_k^\vee+p_{k+1}a_{k+1}^\vee+\cdots+p_la_l^\vee\in\Phi_{re,+}}}e^{-2\pi\pair{y, n_1a_1^\vee+\cdots+n_ka_k^\vee+p_{k+1}a_{k+1}^\vee+\cdots+p_la_l^\vee}}\\&\geq\sum_{p_{k+1},\cdots,p_l\in\BZ_{\geq0}}c(p_{k+1},\cdots,p_l)e^{-2\pi\pair{y,p_{k+1}a_{k+1}^\vee+\cdots+p_la_l^\vee}}
		\end{align*} where $c(p_{k+1},\cdots,p_l)$ is the cardinality of the set $$\{(n_1,\cdots,n_k)\in\BZ_{\geq0}^k:n_1a_1^\vee+\cdots+n_ka_k^\vee+p_{k+1}a_{k+1}^\vee+\cdots+p_la_l^\vee\in\Phi_{re,+}\}$$ which is of polynomial growth in $p_{k+1},\cdots,p_l$ by lemma 3.2.1 (1). So there exists $\delta>0,N>0$ s.t. for $p_{k+1},\cdots,p_l>N$, we have $$c(p_{k+1},\cdots,p_l)e^{-2\pi\pair{y,p_{k+1}a_{k+1}^\vee+\cdots+p_la_l^\vee}}\leq e^{-2\pi(1-\delta)\pair{y,p_{k+1}a_{k+1}^\vee+\cdots+p_la_l^\vee}}$$ for all $y\in R$ and the summation of RHS is convergent because it is a product of some geometric progressions. Hence the result. 
\end{proof}
\begin{cor}
	Let $X\subseteq\Delta$ be a $W$-finite subset. For $t\in\BC,h=x+iy\in V+i\Gamma_X$, the infinite product $$\prod_{a\in\Phi_{re,+}-\Phi_{X,+}}\frac{1-te^{2\pi i\pair{h,a^\vee}}}{1-e^{2\pi i\pair{h,a^\vee}}}$$ as a function in $t$ and $h$ is absolutely convergent uniformly on compact subsets of $\BC\times(V+i\Gamma_X)$. 
\end{cor}
\begin{proof}
	If $1-e^{2\pi i\pair{h,a^\vee}}=0$, then $\pair{h,a^\vee}=0$, so $a^\vee\in\Phi_{X,+}$. After excluding all the roots in $\Phi_{X,+}$, the denominator will never be zero. The rest is done by proposition 3.3.1. 
\end{proof}
\begin{lem}
	Let $X\subseteq\Delta$ be a $W$-finite subset. Let $r$ be the radius of convergence of $W(t)$, $D_r:=\{z\in\BC:|z|<r\}$. For $t\in D_r$ and $h=x+iy\in V+i\Gamma_X$, the sum $$\sum_{w\in W}\prod_{b\in\Phi(w)}\frac{t-e^{2\pi i\pair{h,b^\vee}}}{1-te^{2\pi i\pair{h,b^\vee}}}$$ as a function in $t$ and $h$ is absolutely convergent uniformly on compact subsets of $D_r\times(V+i\Gamma_X)$.
\end{lem}
\begin{proof}
	Any compact subset of $D_r\times (V+i\Gamma_X)$ is contained in a set of the form $\ol D_{r_0}\times (V+iR)$ for $0<r_0<r$ and $R\subset\Gamma_X$ where $\ol D_{r_0}=\{t\in\BC:|t|\leq r_0\}$. We will prove uniform absolutely convergence on this set.
	\par Take $r_1\in(r_0,r)$. Since $\lim_{z\to0}\frac{t-z}{1-tz}=t$, there exists $c>0$ s.t. for all $|z|<c$ and $t\in \ol D_{r_0}$, we have $$\left|\frac{t-z}{1-tz}\right|<r_1$$ Take $N=-\frac{\log c}{2\pi}$, by lemma 3.1.1 (2), let $S\subseteq\Phi_{re,+}$ be a finite set of positive roots s.t. $\pair{y,b^\vee}>N$ for all $b\in\Phi_{re,+}-S$ and $y\in R$. Then for $b\in\Phi_{re,+}-S$ and $y\in R$ we have $|e^{2\pi i\pair{x,b^\vee}-2\pi \pair{y,b^\vee}}|<c$ and  $$\left|\frac{t-e^{2\pi i\pair{h,b^\vee}}}{1-te^{2\pi i\pair{h,b^\vee}}}\right|=\left|\frac{t-e^{2\pi i\pair{x,b^\vee}-2\pi\pair{y,b^\vee}}}{1-te^{2\pi i\pair{x,b^\vee}-2\pi \pair{y,b^\vee}}}\right|<r_1$$
	\par Since $S$ is finite, there exists $A>0$ s.t. for all $b\in S$ we have $$\left|\frac{t-e^{2\pi i\pair{h,b^\vee}}}{1-te^{2\pi i\pair{h,b^\vee}}}\right|\leq A$$ So when $\ell(w)>S$, the term $$\prod_{b\in\Phi(w)}\left|\frac{t-e^{2\pi i\pair{x,b^\vee}-2\pi\pair{y,b^\vee}}}{1-te^{2\pi i\pair{x,b^\vee}-2\pi \pair{y,b^\vee}}}\right|$$ is bounded by $$r_1^{\ell(w)-|S|}A^{|S|}=r_1^{-|S|}A^{|S|}r_1^{\ell(w)}$$ uniformly for $|t|\leq r_0$ and $y\in R$, since each term is a product of $\ell(w)$ fractions, all but $|S|$ of them are smaller than $r_0$, the rest are smaller than $1$. The summation $$\sum_{w\in W}r_1^{\ell(w)}$$ is convergent since $0<r_1<r$. Thus the summation $$\sum_{w\in W,\ell(w)>|S|}\prod_{b\in\Phi(w)}\left|\frac{t-e^{2\pi i\pair{x,b^\vee}-2\pi\pair{y,b^\vee}}}{1-te^{2\pi i\pair{x,b^\vee}-2\pi \pair{y,b^\vee}}}\right|$$ is  uniformly convergent on $\ol D_{r_0}\times (V+iR)$ by Weierstrass test. 
\end{proof}
\par For each $W$-finite subset $X\subseteq\Delta$, we introduce the following function $$\CC_X(t,h)=W_X(t)\left(\sum_{w\in W^X}\prod_{b\in\Phi(w)}\frac{t-e^{2\pi i\pair{h,b^\vee}}}{1-te^{2\pi i\pair{h,b^\vee}}}\right)\left(\prod_{b\in\Phi_{re,+}-\Phi_{X,+}}\frac{1-te^{2\pi i\pair{h,b^\vee}}}{1-e^{2\pi i\pair{h,b^\vee}}}\right)$$ By corollary 3.3.2 and lemma 3.3.3 this is a well-defined function on $D_r\times(V+i\Gamma_X)$. Note that for $X=\phi$, we have $\CC_X(t,h)=\CC(t,h)$. 
\begin{prop}
	Let $X$ be a $W$-finite subset of $\Delta$. For $t\in D_r$ and $h\in V+i\Gamma_X$, the summation $$\sum_{w\in W^X}\prod_{a\in\Phi_{re,+}-\Phi_{X,+}}\frac{1-te^{2\pi i\pair{h,wa^\vee}}}{1-e^{2\pi i\pair{h,wa^\vee}}}$$ is absolutely convergent uniformly on compact subsets of $D_r\times(V+i\Gamma_X)$, and is equal to $$\left(\sum_{w\in W^X}\prod_{b\in\Phi(w)}\frac{t-e^{2\pi i\pair{h,b^\vee}}}{1-te^{2\pi i\pair{h,b^\vee}}}\right)\left(\prod_{b\in\Phi_{re,+}-\Phi_{X,+}}\frac{1-te^{2\pi i\pair{h,b^\vee}}}{1-e^{2\pi i\pair{h,b^\vee}}}\right)$$
\end{prop}
\begin{proof}
	By the definition of $W^X$ (see proposition 2.3.1), for any $w\in W^X,a\in X$ we have $wa>0$, so we have $w\Phi_{X,+}=\Phi_{X,+}$, namely $\Phi(w)\cap\Phi_{X,+}=\phi$, then by similar formal computation as in (1.2), we have \begin{align*}
		&\sum_{w\in W^X}\prod_{a\in\Phi_{re,+}-\Phi_{X,+}}\frac{1-te^{2\pi i\pair{h,wa^\vee}}}{1-e^{2\pi i\pair{h,wa^\vee}}}\\&=\sum_{w\in W^X}\left(\prod_{a\in\Phi_{re,+}-\Phi_{X,+},wa<0}\frac{1-te^{2\pi i\pair{h,wa^\vee}}}{1-e^{2\pi i\pair{h,wa^\vee}}}\right)\left(\prod_{a\in\Phi_{re,+}-\Phi_{X,+},wa>0}\frac{1-te^{2\pi i\pair{h,wa^\vee}}}{1-e^{2\pi i\pair{h,wa^\vee}}}\right)\\&=\sum_{w\in W^X}\left(\prod_{b\in\Phi_{re,+}-\Phi_{X,+},w^{-1}b<0}\frac{1-te^{-2\pi i\pair{h,b^\vee}}}{1-e^{-2\pi i\pair{h,b^\vee}}}\right)\left(\prod_{b\in\Phi_{re,+}-\Phi_{X,+},w^{-1}b>0}\frac{1-te^{2\pi i\pair{h,b^\vee}}}{1-e^{2\pi i\pair{h,b^\vee}}}\right)\\&=\sum_{w\in W^X}\left(\prod_{b\in\Phi(w)}\frac{1-te^{-2\pi i\pair{h,b^\vee}}}{1-e^{-2\pi i\pair{h,b^\vee}}}\cdot\frac{1-e^{2\pi i\pair{h,b^\vee}}}{1-te^{2\pi i\pair{h,b^\vee}}}\right)\left(\prod_{b\in\Phi_{re,+}-\Phi_{X,+}}\frac{1-te^{2\pi i\pair{h,b^\vee}}}{1-e^{2\pi i\pair{h,b^\vee}}}\right)\\&=\left(\sum_{w\in W^X}\prod_{b\in\Phi(w)}\frac{t-e^{2\pi i\pair{h,b^\vee}}}{1-te^{2\pi i\pair{h,b^\vee}}}\right)\left(\prod_{b\in\Phi_{re,+}-\Phi_{X,+}}\frac{1-te^{2\pi i\pair{h,b^\vee}}}{1-e^{2\pi i\pair{h,b^\vee}}}\right)
	\end{align*} since RHS is absolutely convergent on $D_r\times(V+i\Gamma_X)$. 
\end{proof}
By proposition 3.3.4, we have $$\CC_X(t,h)=W_X(t)\sum_{w\in W^X}\prod_{a\in\Phi_{re,+}-\Phi_{X,+}}\frac{1-te^{2\pi i\pair{h,wa^\vee}}}{1-e^{2\pi i\pair{h,wa^\vee}}}$$ 
\begin{prop}
	For $t\in D_r,h\in V+iC$ we have $\CC_X(t,h)=\CC(t,h)$. 
\end{prop}
\begin{proof}
	Because $\CC(t,h)$ is absolutely convergence for $h\in V+iC$, we can change the order of summation and change the order in each infinite product, namely we have \begin{align*}
		&\sum_{w\in W}\prod_{a\in\Phi_{re,+}}\frac{1-te^{2\pi i\pair{h,wa^\vee}}}{1-e^{2\pi i\pair{h,wa^\vee}}}=\sum_{w_1\in W^X}\sum_{w_2\in W_X}\prod_{a\in\Phi_{re,+}}\frac{1-te^{2\pi i\pair{h,w_1w_2a^\vee}}}{1-e^{2\pi i\pair{h,w_1w_2a^\vee}}}\\&=\sum_{w_1\in W^X}\sum_{w_2\in W_X}\left(\prod_{a\in\Phi_{X,+}}\frac{1-te^{2\pi i\pair{h,w_1w_2a^\vee}}}{1-e^{2\pi i\pair{h,w_1w_2a^\vee}}}\right)\cdot\left(\prod_{a\in\Phi_{re,+}-\Phi_{X,+}}\frac{1-te^{2\pi i\pair{h,w_1w_2a^\vee}}}{1-e^{2\pi i\pair{h,w_1w_2a^\vee}}}\right)
	\end{align*}Since $w_2\in W_X$, $\Phi(w_2)\subseteq\Phi_X$, namely $w_2$ permutes the positive roots in $\Phi_{re,+}-\Phi_{X,+}$. Thus the above is equal to 
		$$\sum_{w_1\in W^X}\left(\prod_{a\in\Phi_{re,+}-\Phi_{X,+}}\frac{1-te^{2\pi i\pair{h,w_1a^\vee}}}{1-e^{2\pi i\pair{h,w_1a^\vee}}}\right)\cdot\left(\sum_{w_2\in W_X}\prod_{a\in\Phi_{X,+}}\frac{1-te^{2\pi i\pair{h,w_1w_2a^\vee}}}{1-e^{2\pi i\pair{h,w_1w_2a^\vee}}}\right)$$
		\par  By the identity in \cite{Mac1}, we have $$\sum_{w'\in W_X}w'\left(\prod_{a\in\Phi_{X,+}}\frac{1-te^{-a^\vee}}{1-e^{-a^\vee}}\right)=W_X(t)$$ So $$\sum_{w_2\in W_X}\prod_{a\in\Phi_{X,+}}\frac{1-te^{2\pi i\pair{h,w_1w_2a^\vee}}}{1-e^{2\pi i\pair{h,w_1w_2a^\vee}}}=W_X(t)$$ is constant in $h$. So we have \begin{align*}
			&\sum_{w_1\in W^X}\left(\prod_{a\in\Phi_{re,+}-\Phi_{X,+}}\frac{1-te^{2\pi i\pair{h,w_1a^\vee}}}{1-e^{2\pi i\pair{h,w_1a^\vee}}}\right)\cdot\left(\sum_{w_2\in W_X}\prod_{a\in\Phi_{X,+}}\frac{1-te^{2\pi i\pair{h,w_1w_2a^\vee}}}{1-e^{2\pi i\pair{h,w_1w_2a^\vee}}}\right)\\&=\sum_{w_1\in W^X}\left(\prod_{a\in\Phi_{re,+}-\Phi_{X,+}}\frac{1-te^{2\pi i\pair{h,w_1a^\vee}}}{1-e^{2\pi i\pair{h,w_1a^\vee}}}\right)\cdot W_X(t)=\CC_X(t,h)
	\end{align*}
\end{proof}
\par Recall that $V+iS_X$ is an open set because $S_X$ is open in $V$. By lemma 2.2.2, we have $$V+iS_X=V+i\bigcup_{w\in W_X}\bigcup_{Y\subseteq X}wF_Y=V+i\bigcup_{w\in W_X}w\Gamma_X$$ is the union of all the $W_X$-translates of the set $V+i\Gamma_X$. Also we have $V+iW\cdot S_X=V+iW\cdot \Gamma_X=W\cdot(V+i\Gamma_X)$ is an open set in $\Omega$. 
\begin{prop}
	  The summation \begin{equation}
	  	\sum_{w\in W^X}\prod_{a\in\Phi_{re,+}-\Phi_{X,+}}\frac{1-te^{2\pi i\pair{h,wa^\vee}}}{1-e^{2\pi i\pair{h,wa^\vee}}}
	  \end{equation} is absolutely convergent on $D_r\times(V+iW\cdot S_X)$. So we can extend the domain of $\CC_X(t,h)$ by defining $$\CC_X(t,h)=W_X(t)\sum_{w\in W^X}\prod_{a\in\Phi_{re,+}-\Phi_{X,+}}\frac{1-te^{2\pi i\pair{h,wa^\vee}}}{1-e^{2\pi i\pair{h,wa^\vee}}}$$ for $(t,h)\in D_r\times(V+iW\cdot S_X)$. Then $\CC_X(t,h)$ is a $W_X$-invariant function on $D_r\times(V+iW\cdot S_X)$. 
\end{prop}
\begin{proof}
	Fix $t\in D_r,h\in V+iF_Y$ for some $Y\subseteq X$. First of all, for $w_X\in W_X$ we have $$\prod_{a\in\Phi_{re,+}-\Phi_{X,+}}\frac{1-te^{2\pi i\pair{h,wa^\vee}}}{1-e^{2\pi i\pair{h,wa^\vee}}}=\prod_{a\in\Phi_{re,+}-\Phi_{X,+}}\frac{1-te^{2\pi i\pair{h,ww_Xa^\vee}}}{1-e^{2\pi i\pair{h,ww_Xa^\vee}}}$$ since $w_X$ permutes the positive roots in $\Phi_{re,+}-\Phi_{X,+}$ as in the proof of the prop 3.3.3. So for $w_1,w_2\in W$ lying in the same left $W_X$-coset, we have$$\prod_{a\in\Phi_{re,+}-\Phi_{X,+}}\frac{1-te^{2\pi i\pair{h,w_1a^\vee}}}{1-e^{2\pi i\pair{h,w_1a^\vee}}}=\prod_{a\in\Phi_{re,+}-\Phi_{X,+}}\frac{1-te^{2\pi i\pair{h,w_2a^\vee}}}{1-e^{2\pi i\pair{h,w_2a^\vee}}}$$ Note that $W^X$ is a set of coset representatives of $W/W_X$, and for any $w'\in W$,  $w'W^X$ is also a set of coset representatives of $W/W_X$, so \begin{align*}
		&\CC_X(t,h)=W_X(t)\sum_{w\in W^X}\prod_{a\in\Phi_{re,+}-\Phi_{X,+}}\frac{1-te^{2\pi i\pair{h,wa^\vee}}}{1-e^{2\pi i\pair{h,wa^\vee}}}\\&=W_X(t)\sum_{w\in w_XW^X}\prod_{a\in\Phi_{re,+}-\Phi_{X,+}}\frac{1-te^{2\pi i\pair{h,wa^\vee}}}{1-e^{2\pi i\pair{h,wa^\vee}}}\\&=W_X(t)\sum_{w\in W^X}\prod_{a\in\Phi_{re,+}-\Phi_{X,+}}\frac{1-te^{2\pi i\pair{h,w'wa^\vee}}}{1-e^{2\pi i\pair{h,w'wa^\vee}}}\\&=W_X(t)\sum_{w\in W^X}\prod_{a\in\Phi_{re,+}-\Phi_{X,+}}\frac{1-te^{2\pi i\pair{w'^{-1}h,wa^\vee}}}{1-e^{2\pi i\pair{w'^{-1}h,wa^\vee}}}
	\end{align*} 
	so RHS is absolutely convergent and is equal to $\CC_X(t,w'^{-1}h)$. By the description of $V+iW\cdot S_X$ before this prop, every element in $V+iW\cdot S_X$ is of the form $w'^{-1}h$ for some $w'\in W,h\in V+iF_Y$ for some $Y\subseteq X$. Thus (3.3) is convergent on $D_r\times(V+iW\cdot S_X)$ and $\CC_X(t,h)$ is $W$-invariant. 
\end{proof}
\begin{prop}
	Let $X$ be a $W$-finite subset of $\Delta$. The absolutely convergence of the summation $$\sum_{w\in W^X}\prod_{a\in\Phi_{re,+}-\Phi_{X,+}}\frac{1-te^{2\pi i\pair{h,wa^\vee}}}{1-e^{2\pi i\pair{h,wa^\vee}}}$$ is uniform on compact subsets of  $D_r\times(V+iS_X)$, thus $\CC_X$ is a holomorphic function on $D_r\times(V+iW\cdot S_X)$. 
\end{prop}
\begin{proof}
	We recall that $V+iS_X=V+i\bigcup_{w\in W_X}w\Gamma_X$. A compact set of $D_r\times (V+iS_X)$ is contained in a compact set of the form $\ol D_{r_0}\times E$ where $0<r_0<r$ and $E\subseteq V+iS_X$ is compact. Since $W_X$ is finite, the set $W_X\cdot E$ is also compact. Let $E_0=(W_X\cdot E)\cap (V+i\Gamma_X)$, then $W_X\cdot E=W_X\cdot E_0$ and $E_0\subseteq V+i\Gamma_X$ is compact by the concrete description of $\Gamma_X$ and $S_X$ after lemma 2.2.2. So the convergence of this summation is uniform on $E_0$ by proposition 3.3.4, and the convergence on $W_X\cdot E_0$ is uniform since $W_X$ is finite and this summation is $W$-invariant. 
	\par The uniformness of convergence on compact subsets of $V+iS_X$ implies that $\CC_X$ is holomorphic on $V+iS_X$. Since $\CC_X$ is $W$-invariant, it is also holomorphic on each $W$-translation of $V+iS_X$, so it is holomorphic on $W\cdot(V+iS_X)=V+iW\cdot S_X$. 
\end{proof}
So each $\CC_X(t,h)$ is a holomorphic continuation of $\CC(t,h)$ from $V+iW\cdot C$ to $V+iW\cdot S_X$. This implies that for any $W$-finite $X,X'\subseteq\Delta$ we have $$\CC_X(t,h)=\CC_{X'}(t,h)$$ for $t\in D_r,h\in V+iW\cdot(S_X\cap S_{X'})$, and all these functions together give a holomorphic continuation of $\CC(t,h)$ to the open set $$D_r\times \bigcup_{X\subseteq\Delta:|W_X|<\infty}(V+iW\cdot S_X)$$ Which is equal to the whole $D_r\times\Omega$ by corollary 2.2.4. So we have proved
\begin{thm}
	$\CC(t,h)$ can be analytically continued to a holomorphic function on $D_r\times\Omega$. 
\end{thm}


\begin{thebibliography}{aa}
\bibitem[BFK]{BFK}Alexander Braverman, Michael Finkelberg and David Kazhdan. "Affine Gindikin–Karpelevich Formula via Uhlenbeck Spaces." Contributions in Analytic and Algebraic Number Theory. New York, NY: Springer New York, 2011. 17-29. Springer Proceedings in Mathematics.\bibitem[BGKP]{BGKP}Alexander Braverman, Howard Garland, David Kazhdan and Manish Patnaik. "An Affine Gindikin-Karpelevich Formula." in: Perspectives in representation theory : a conference in honor of Igor Frenkel's 60th birthday on perspectives in representation theory, May 12-17, 2012, Yale University, New Haven, CT.
\bibitem[BKP]{BKP}Alexander Braverman, David Kazhdan and Manish M.Patnaik. "Iwahori–Hecke Algebras for P-adic Loop Groups." Inventiones Mathematicae 204.2 (2016): 347-442.
\bibitem[BPGR1]{BPGR1}Nicole Bardy-Panse, Stéphane Gaussent and Guy Rousseau. "
Iwahori–Hecke algebras for Kac–Moody groups over local fields." Pacific J. Math. 285 (2016) 1-61.
\bibitem[BPGR2]{BPGR2}Nicole Bardy‐Panse, Stéphane Gaussent and Guy Rousseau. "Macdonald's Formula for Kac–Moody Groups over Local Fields." Proceedings of the London Mathematical Society 119.1 (2019): 135-175. 
\bibitem[Ch]{Ch}Ivan Cherednik. "Double Affine Hecke Algebras and Macdonald's Conjectures." Annals of Mathematics 141.1 (1995): 191-216.
\bibitem[Kac]{Kac}Victor G.Kac: "Infinite Dimensional Lie Algebras." Cambridge University Press, 1990. 
\bibitem[KP]{KP}Victor G Kač, Dale H Peterson. "Infinite-dimensional Lie algebras, theta functions and modular forms." Adv. in Math. 53-2 (1984): 125-264. 
\bibitem[LLO]{LLO}Kyu-Hwan Lee, Dongwen Liu and Thomas Oliver. "Character expansion of Kac–Moody correction factors."  Pacific J. Math. 313 (2021) No.1: 159-183.
\bibitem[Lo]{Lo}Eduard Looijenga. "Invariant Theory for Generaized Root Systems." Inventiones Mathematicae 61.1 (1980): 1-32. 
\bibitem[Mac1]{Mac1}I.G.Macdonald. "The Poincare Series of a Coxeter Group." Mathematische Annalen 199.3 (1972): 161-174.
\bibitem[Mac2]{Mac2}I.G.Macdonald. "A Formal Identity for Affine Root Systems. " Lie Groups and Symmetric Spaces: In Memory of F.I. Karpelevich. Amer. Math. Soc. Transl. (2) Vol. 210 (2003): 195-211. 
\bibitem[Mac3]{Mac3}I.G.MacDonald. "Some Conjectures for Root Systems." SIAM Journal on Mathematical Analysis 13.6 (1982): 988-1007.
\bibitem[MPW]{MPW}Dinakar Muthiah, Anna Puskás and Ian Whitehead. "Correction Factors for Kac–Moody Groups and T-deformed Root Multiplicities." Mathematische Zeitschrift 296.1-2 (2020): 127-145.
\bibitem[P]{P}Manish M.Patnaik. "Unramified Whittaker Functions on P-adic Loop Groups." American Journal of Mathematics 139.1 (2017): 175-213.
\bibitem[PP]{PP}Manish M.Patnaik and Anna Puskás. "Metaplectic Covers of Kac–Moody Groups and Whittaker Functions." Duke Mathematical Journal 168.4 (2019): Duke Mathematical Journal, 2019-03-15, Vol.168 (4).
\bibitem[St]{St}Robert Steinberg. “Endomorphisms of linear algebraic groups.” Memoirs of the American Mathematical Society (1968). 
\end{thebibliography}
\end{document}